\documentclass[11pt]{article}
\usepackage{amsmath}
\usepackage{amssymb}
\usepackage{amsthm}
\usepackage{enumerate}
\usepackage[colorlinks]{hyperref}
\usepackage{enumitem}
\usepackage[margin=1in]{geometry}
\usepackage{stmaryrd}
\usepackage{pgfplots, pgfplotstable, booktabs, colortbl, array,multirow}
\usepgfplotslibrary{groupplots}
\pgfplotsset{compat=newest}

\usepackage[style=numeric-comp,firstinits=true,url=false,maxbibnames=5,backend=bibtex]{biblatex}
\bibliography{references.bib}

\theoremstyle{plain}
\newtheorem{theorem}{Theorem}[section]
\newtheorem{lemma}[theorem]{Lemma}
\newtheorem{proposition}[theorem]{Proposition}

\theoremstyle{definition}

\newtheorem{remark}{Remark}[section]

\DeclareMathOperator{\dv}{div}
\DeclareMathOperator{\curl}{curl}

\newcommand{\textoverline}[1]{$\overline{\mbox{#1}}$}

\newcommand{\maketable}[1]{%
\pgfplotstabletypeset[
header=false,
font=\small,
every head row/.style={before row=\hline,after row=\hline},
every last row/.style={after row=\hline},
every row no 4/.style={before row=\hline},
every row no 8/.style={before row=\hline},
every row no 12/.style={before row=\hline},
create on use/newcol/.style={string type,
        create col/set list={,$($\ref{velocityh_simp}$\;$,-\ref{Eh_simp}$)$,,,$($\ref{velocityh_simp}$\;$,-\ref{Eh_simp}$)$,,,$($\ref{velocityh_comp}$\;$,-\ref{alphah_comp}$)$,,,$($\ref{velocityh_comp}$\;$,-\ref{alphah_comp}$)$,},
        },
create on use/newcol2/.style={
        create col/set list={,No,,,,Yes,,,,No,,,,Yes,,}
		},
columns={newcol,newcol2,0,1,2,3,4,5,6,7,8},
columns/newcol/.style={string type,column type/.add={|}{},column name={Eqn.}},
columns/newcol2/.style={string type,column type/.add={|}{},column name={Upwind}},
columns/0/.style={sci zerofill,column type/.add={|}{|},column name={$h^{-1}$}},
columns/1/.style={dec sep align={c|},sci,sci 10e,sci zerofill,precision=2,column type/.add={}{|},column name={$\|u_h-u\|$}},
columns/3/.style={dec sep align={c|},sci,sci 10e,sci zerofill,precision=2,column type/.add={}{|},column name={$\|B_h-B\|$}}, 
columns/5/.style={dec sep align={c|},sci,sci 10e,sci zerofill,precision=2,column type/.add={}{|},column name={$\|\rho_h-\rho\|$}}, 
columns/7/.style={dec sep align={c|},sci,sci 10e,sci zerofill,precision=2,column type/.add={}{|},column name={$\|p_h-p\|$}}, 
columns/2/.style={clear infinite,dec sep align={c|},fixed zerofill,precision=2,column type/.add={}{|},column name={Rate}},
columns/4/.style={clear infinite,dec sep align={c|},fixed zerofill,precision=2,column type/.add={}{|},column name={Rate}},
columns/6/.style={clear infinite,dec sep align={c|},fixed zerofill,precision=2,column type/.add={}{|},column name={Rate}},
columns/8/.style={clear infinite,dec sep align={c|},fixed zerofill,precision=2,column type/.add={}{|},column name={Rate}},
]
{#1}
}

\newcommand{\makeplot}[1]{
\addplot[blue,thick] table [x expr=\coordindex*0.02, y expr=abs(\thisrowno{0})]{#1};
\addplot[red,thick] table [x expr=\coordindex*0.02, y expr=abs(\thisrowno{1})]{#1};
\addplot[black,thick] table [x expr=\coordindex*0.02, y expr=abs(\thisrowno{2})]{#1};
\addplot[blue,dashed,thick] table [x expr=\coordindex*0.02, y expr=abs(\thisrowno{4})]{#1};
\addplot[red,dashed,thick] table [x expr=\coordindex*0.02, y expr=abs(\thisrowno{5})]{#1};
\addplot[black,dashed,thick] table [x expr=\coordindex*0.02, y expr=abs(\thisrowno{6})]{#1};
}

\newcommand{\makeplotone}[1]{
\addplot[green,thick] table [x expr=\coordindex*0.02, y expr=abs(\thisrowno{3})]{#1};
\addplot[black,thick] table [x expr=\coordindex*0.02, y expr=abs(\thisrowno{2})]{#1};
\addplot[blue,dashed,thick] table [x expr=\coordindex*0.02, y expr=abs(\thisrowno{4})]{#1};
\addplot[red,dashed,thick] table [x expr=\coordindex*0.02, y expr=abs(\thisrowno{5})]{#1};
\addplot[black,dashed,thick] table [x expr=\coordindex*0.02, y expr=abs(\thisrowno{6})]{#1};
}

\begin{document}

\title{A Finite Element Method for MHD that Preserves Energy, Cross-Helicity, Magnetic Helicity, Incompressibility, and $\dv B = 0$}

\author{Evan S. Gawlik\thanks{\noindent Department of Mathematics,  University of Hawai`i at M\textoverline{a}noa, \href{egawlik@hawaii.edu}{egawlik@hawaii.edu}} \; and \; Fran\c{c}ois Gay-Balmaz\thanks{\noindent CNRS - LMD, Ecole Normale Sup\'erieure, \href{francois.gay-balmaz@lmd.ens.fr}{francois.gay-balmaz@lmd.ens.fr}}}

\date{}

\maketitle

\begin{abstract}
We construct a structure-preserving finite element method and time-stepping scheme for inhomogeneous, incompressible magnetohydrodynamics (MHD).  The method preserves energy, cross-helicity (when the fluid density is constant), magnetic helicity, mass, total squared density, pointwise incompressibility, and the constraint $\dv B = 0$ to machine precision, both at the spatially and temporally discrete levels.
\end{abstract}

\section{Introduction} \label{sec:intro}

In this paper, we construct a structure-preserving finite element method for solving the inhomogeneous, incompressible magnetohydrodynamic (MHD) equations on a bounded domain $\Omega \subset \mathbb{R}^d$, $d \in \{2,3\}$.  These equations seek a velocity field $u$, magnetic field $B$, pressure $p$, and density $\rho$ satisfying
\begin{align}
\rho(\partial_t u + u \cdot \nabla u) - (\nabla \times B) \times B &= -\nabla p, & \text{ in } \Omega \times (0,T), \label{velocity0} \\
\partial_t B - \nabla \times (u \times B) &= 0, & \text{ in } \Omega \times (0,T), \label{magnetic0} \\
\partial_t \rho + \dv (\rho u) &= 0, & \text{ in } \Omega \times (0,T), \label{density0} \\
\dv u = \dv B &= 0, & \text{ in } \Omega \times (0,T), \label{incompressible0} \\
u \cdot n = B \cdot n &= 0, & \text{ on } \partial\Omega \times (0,T), \label{BC} \\
u(0) = u_0, \, B(0) = B_0, \,  \rho(0) &= \rho_0, & \text{ in } \Omega. \label{IC}
\end{align}
The method we construct exactly preserves energy $\frac{1}{2}\int_\Omega \rho u \cdot u + B \cdot B \, dx$, cross-helicity $\int_\Omega u \cdot B \, dx$ (when $\rho \equiv 1$), magnetic helicity $\int_\Omega A \cdot B \, dx$, mass $\int_\Omega \rho \, dx$, total squared density $\int_\Omega \rho^2 \, dx$, and the constraints $\dv u = \dv B = 0$ at the spatially and temporally discrete level.  Here, $A$ denotes the magnetic potential; that is, $A$ is any vector field satisfying $\nabla \times A = B$ and $\left. A \times n \right|_{\partial\Omega} = 0$.

Our method builds upon a growing body of literature on structure preservation in incompressible MHD simulations.  Much of this literature focuses on the setting of constant density.  In that setting, researchers have constructed energy-stable schemes that preserve $\dv B=0$~\cite{hu2017stable}; energy-stable schemes that preserve $\dv u = \dv B = 0$~\cite{hiptmair2018fully}; schemes that preserve energy, cross-helicity, and $\dv u = \dv B=0$~\cite{liu2001energy,gawlik2011geometric}; and schemes that preserve energy, cross-helicity, $\int_\Omega A \, dx$, and $\dv u = \dv B = 0$ in two dimensions~\cite{kraus2017variational}.  
More recently, Hu, Lee, and Xu~\cite{hu2020helicity} constructed a finite element method for homogeneous, incompressible MHD that preserves energy, cross-helicity, magnetic helicity, and $\dv B = 0$.

Our method resembles the one proposed by Hu, Lee, and Xu~\cite{hu2020helicity}, but it differs in several key respects:
\begin{enumerate}
\item We treat the boundary conditions $\left. u \cdot n \right|_{\partial\Omega}=0$, whereas~\cite{hu2020helicity} treats the boundary conditions $\left. u \times n \right|_{\partial\Omega} = 0$.
\item Our method produces a velocity field $u$ satisfying $\dv u = 0$ pointwise in $\Omega$, whereas the computed velocity field in~\cite{hu2020helicity} only obeys this constraint in a weak sense.
\item We allow the density $\rho$ to be variable.  This introduces novel challenges, since the conserved energy $\frac{1}{2}\int_\Omega \rho u \cdot u + B \cdot B \, dx$ no longer depends quadratically on the unknowns $u$, $B$, $\rho$.  We overcome this difficulty by carefully selecting a weak formulation of~(\ref{velocity0}-\ref{incompressible0}) to discretize spatially, and by designing a time discretization that is similar but not identical to the midpoint rule.  We also show how to incorporate upwinding in the density advection without sacrificing any conservation laws other than $\int_\Omega \rho^2 \, dx$.
\end{enumerate}

Some of the techniques we use in this paper to achieve conservation of invariants in the discrete setting are adapted from our earlier work on conservative methods for the incompressible Euler equations with variable density~\cite{gawlik2020conservative}.  Our choice of weak formulation is one example.  We describe our weak formulation of~(\ref{velocity0}-\ref{IC}) in Section~\ref{sec:weak}, following closely the presentation in~\cite{gawlik2020conservative}.  We also adopt a generalization of~\cite{gawlik2020conservative}'s temporal discretization.  As observed there, a useful way to achieve energy conservation in the presence of variable density is to use the midpoint rule for all terms except one involving $u \cdot u$, which is discretized as $u_k \cdot u_{k+1}$ when stepping from time $t_k$ to $t_{k+1}$.  See Section~\ref{sec:temp} for details.  A point where we deviate from~\cite{gawlik2020conservative} is in our spatial discretization of the momentum advection term $\rho u \cdot \nabla u$.  Here, ensuring cross-helicity conservation and $\dv B=0$ requires us to adopt a different discretization of the momentum advection term than in~\cite{gawlik2020conservative}.  

We present our numerical method in dimension $d=3$, but it is straightforward to adapt our setup to dimension $d=2$; see Remark~\ref{remark:dim2}.  Note that in dimension $d=2$, magnetic helicity conservation is automatic if the constraint $\dv B =0$ holds pointwise.  This is because we may take the magnetic potential $A$ to be a vector field orthogonal to the plane containing $\Omega$ in two dimensions.  For this reason, we present in this paper two methods in dimension $d=3$: one that preserves all of the above invariants, and one that preserves all but magnetic helicity.  Both methods preserve all invariants when reduced to two dimensions, but the latter is a slightly simpler method.

This paper is organized as follows.  We start in Section~\ref{sec:weak} by writing down a weak formulation of~(\ref{velocity0}-\ref{IC}) and studying its invariants of motion.  We propose a spatial discretization in Section~\ref{sec:space}, focusing first on one that preserves all invariants except magnetic helicity.  We present an alternative spatial discretization that also preserves magnetic helicity in Section~\ref{sec:maghelicity}.  We describe how to incorporate upwinding in Section~\ref{sec:upwinding}, and we propose a temporal discretization in Section~\ref{sec:temp}.  We conclude with numerical examples in Section~\ref{sec:numerical}.

\section{Weak Formulation and Conserved Quantities} \label{sec:weak}

In this section, we derive a weak formulation of~(\ref{velocity0}-\ref{IC}) and study its invariants of motion.

Following~\cite{gawlik2020conservative}, we use the identity
\[
\rho u \cdot \nabla u = \nabla(\rho u \cdot u) - u \times (\nabla \times (\rho u)) - (u \cdot \nabla \rho) u - \frac{1}{2} \rho \nabla ( u \cdot u)
\]
and equations~(\ref{density0}-\ref{incompressible0}) to write~(\ref{velocity0}) in the form
\begin{align}
\partial_t (\rho u) + (\nabla \times (\rho u)) \times u - (\nabla \times B) \times B - \frac{1}{2}\rho \nabla(u \cdot u)   &= -\nabla \widetilde{p}, \label{velocity1}
\end{align}
where $\widetilde{p} = p + \rho u \cdot u$.
Next, we multiply~(\ref{velocity1}),~(\ref{magnetic0}),~(\ref{density0}), and~(\ref{incompressible0}) by test functions/vector fields $v$, $C$, $\sigma$, and $q$, respectively, and integrate by parts.  Using the identity
\begin{align} \label{IBP}
\int_\Omega \left( (\nabla \times w) \times u \right) \cdot v \, dx = \int_\Omega w \cdot \nabla \times (u \times v) \, dx, \quad \text{ if } \left.u \cdot n\right|_{\partial\Omega} = \left.v \cdot n\right|_{\partial\Omega} = 0,
\end{align}
we deduce the following.  For every pair of smooth vector fields $v$ and $C$ satisfying $\left. v \cdot n \right|_{\partial\Omega} = \left. C \cdot n \right|_{\partial\Omega} = 0$ and every pair of smooth scalar fields $\sigma$ and $q$, the solution $(u,B,\rho,\widetilde{p})$ of~(\ref{velocity0}-\ref{IC}) satisfies
\begin{align}
\langle \partial_t(\rho u), v \rangle + a(\rho u, u, v) - a(B,B,v) + \frac{1}{2} b(u \cdot u,\rho,v) &= \langle \widetilde{p}, \dv v \rangle, \label{velocity} \\
\langle \partial_t B, C \rangle + a(C,B,u) &= 0, \label{magnetic} \\
\langle \partial_t \rho, \sigma \rangle + b(\sigma,\rho,u) &= 0, \label{density} \\
\langle \dv u, q \rangle &= 0, \label{incompressible}
\end{align}
where $\langle u, v \rangle = \int_\Omega u \cdot v \, dx$ for vector fields $u$ and $v$, $\langle f,g \rangle = \int_\Omega f g \, dx$ for scalar fields $f$ and $g$, and
\begin{align*}
a(w,u,v) &= \langle w, \nabla \times ( u \times v ) \rangle, \\
b(f,g,w) &= -\langle w \cdot \nabla f, g \rangle.
\end{align*}
\begin{remark}
The structure of equations~(\ref{velocity}-\ref{incompressible}) is made even more transparent if one introduces the Lagrangian $\ell(u,B,\rho) = \frac{1}{2} \langle \rho u, u \rangle - \frac{1}{2}\langle B, B \rangle$ of inhomogeneous, incompressible MHD.  In terms of $\frac{\delta \ell}{\delta u} = \rho u$, $\frac{\delta \ell}{\delta B} = -B$, and $\frac{\delta \ell}{\delta \rho} = \frac{1}{2} u \cdot u$, equations~(\ref{velocity}-\ref{incompressible}) take the form
\begin{align}
\left\langle  \partial_t \frac{\delta \ell}{\delta u} , v \right\rangle  + a\left(\frac{\delta \ell}{\delta u}, u, v\right) + a\left(\frac{\delta \ell}{\delta B},B,v\right) + b\left( \frac{\delta \ell}{\delta \rho},\rho,v \right) &= \langle \widetilde{p}, \dv v \rangle, \label{velocity_lag} \\
\langle \partial_t B, C \rangle + a(C,B,u) &= 0, \label{magnetic_lag} \\
\langle \partial_t \rho, \sigma \rangle + b(\sigma,\rho,u) &= 0, \label{density_lag} \\
\langle \dv u, q \rangle &= 0. \label{incompressible_lag}
\end{align}
It is this variational structure that inspired the numerical method we propose in this paper.  We refer the reader to~\cite{gawlik2019variational} for more background.
\end{remark}

The formulation~(\ref{velocity}-\ref{incompressible}) allows one to easily deduce its invariants of motion from basic properties of the trilinear forms $a$ and $b$.  
Namely, $a$ is alternating in its last two arguments,
\begin{equation} \label{aalternating}
a(w,u,v) = -a(w,v,u), 
\end{equation}
and $b$ is alternating in its first two arguments when its last argument is divergence-free:
\begin{align} \label{balternating}
b(f,g,w) &= -b(g,f,w) \text{ if } \dv w = 0 \text{ and } \left. w \cdot n \right|_{\partial\Omega} = 0.
\end{align}
Also,
\begin{equation} \label{avanishes}
a(w,u,v) = 0 \text{ if }  \left.u \cdot n\right|_{\partial\Omega} = \left.v \cdot n\right|_{\partial\Omega} = 0 \text{ and } \nabla \times w = u,
\end{equation}
owing to~(\ref{IBP}).

These properties, together with more elementary ones, give rise to the following conservation laws.  We deduce conservation of mass by taking $\sigma = 1$ in the density equation~(\ref{density}):
\[
\frac{d}{dt} \int_\Omega \rho \, dx = \langle \partial_t \rho, 1 \rangle = -b(1,\rho,u) = 0.
\]
If instead we take $\sigma = \rho$ in~(\ref{density}) and use~(\ref{balternating}), we deduce conservation of total squared density:
\[
\frac{d}{dt} \frac{1}{2} \int_\Omega \rho^2 \, dx = \langle \partial_t \rho, \rho \rangle = -b(\rho,\rho,u) = 0.
\]
Taking $v=u$ in the momentum equation~(\ref{velocity}) and $C=B$ in the magnetic field equation~(\ref{magnetic}) gives conservation of energy:
\begin{align*}
\frac{1}{2} \frac{d}{dt} \int_\Omega \rho u \cdot u &+ B \cdot B \, dx 
= \langle \partial_t(\rho u), u \rangle - \frac{1}{2} \langle \partial_t \rho, u \cdot u \rangle + \langle \partial_t B, B \rangle \\
&= \langle \widetilde{p}, \dv u \rangle - a(\rho u, u, u) + a(B,B,u) - \frac{1}{2}b(u \cdot u, \rho, u) - \frac{1}{2} \langle \partial_t \rho, u \cdot u \rangle - a(B,B,u)  \\
&= 0.
\end{align*}
Here, we have used the fact that $\dv u = 0$, $a$ is alternating in its last two arguments, and~(\ref{density}) holds.

If $\rho \equiv 1$, then taking $v=B$ in the momentum equation~(\ref{velocity}) and $C=u$ in the magnetic field equation~(\ref{magnetic}) gives conservation of cross-helicity:
\begin{align*}
\frac{d}{dt} \int_\Omega u \cdot B \, dx 
&= \langle \partial_t u, B \rangle + \langle \partial_t B, u \rangle \\
&= \langle \widetilde{p}, \dv B \rangle - a(u, u, B) + a(B,B,B) - \frac{1}{2}b(u \cdot u, 1, B) - a(u,B,u) \\
&= 0.
\end{align*}
The last line above follows from the fact that $\dv B = 0$, $b(u \cdot u, 1, B) = -b(1, u \cdot u, B) = 0$, and $a$ is alternating in its last two arguments. 

Finally, if $A$ is any vector field satisfying $\nabla \times A = B$ and $\left. A \times n \right|_{\partial\Omega} = 0$, then conservation of magnetic helicity follows from
\begin{align*}
\frac{d}{dt} \int_\Omega A \cdot B \, dx 
&= \langle \partial_t A, B \rangle + \langle A, \partial_t B \rangle \\
&= \langle \partial_t A, \nabla \times A \rangle + \langle A, \partial_t B \rangle \\
&= \langle \nabla \times (\partial_t A), A \rangle + \langle A, \partial_t B \rangle \\
&= \langle  \partial_t B, A \rangle + \langle A, \partial_t B \rangle \\
&= -2 a(A,B,u) \\
&= 0.
\end{align*}
Here, we have used the magnetic field equation~(\ref{magnetic}) and the property~(\ref{avanishes}) of $a$.

\section{Spatial Discretization} \label{sec:space}

To construct a spatial discretization of~(\ref{velocity}-\ref{incompressible}) that preserves the invariants discussed in Section~\ref{sec:weak}, we will design discretizations of the trilinear forms $a$ and $b$ that satisfy analogues of~(\ref{aalternating}),~(\ref{balternating}), and~(\ref{avanishes}).  By a careful choice of finite element spaces, the method we construct will also preserve the constraints $\dv u = 0$ and $\dv B = 0$ pointwise.

To simplify the presentation, we first describe a spatial discretization that preserves all of the invariants mentioned above except for the magnetic helicity $\int_\Omega A \cdot B \, dx$.  For a method that also preserves magnetic helicity, see Section~\ref{sec:maghelicity}.

We will make use of the following function spaces:
\begin{align*}
H^1_0(\Omega) &= \{f \in L^2(\Omega) \mid \nabla f \in L^2(\Omega)^d, \, f=0 \text{ on } \partial\Omega \}, \\
H_0(\curl,\Omega) &= 
\begin{cases}
\{ u \in L^2(\Omega)^2 \mid \partial_x u_y - \partial_y u_x \in L^2(\Omega), \, u_x n_y - u_y n_x = 0 \text{ on } \partial\Omega \}, &\mbox{ if } d = 2, \\
\{ u \in L^2(\Omega)^3 \mid \curl u \in L^2(\Omega)^3, \, u \times n = 0 \text{ on } \partial\Omega \}, &\mbox{ if } d = 3, \\
\end{cases} \\
H_0(\dv,\Omega) &= \{u \in L^2(\Omega)^d \mid \dv u \in L^2(\Omega), \, u \cdot n = 0 \text{ on } \partial\Omega \}, \\
\mathring{H}(\dv,\Omega) &= \{u \in H_0(\dv,\Omega) \mid \dv u = 0\}, \\
L^2_{\int=0}(\Omega) &= \{f \in L^2(\Omega) \mid \textstyle\int_\Omega f \, dx = 0 \}.
\end{align*}

Let $\mathcal{T}_h$ be a triangulation of $\Omega$, and let $\mathcal{E}_h$ denote the set of interior $(d-1)$-dimensional faces in $\mathcal{T}_h$. For each integer $s \ge 0$ and each simplex $K \in \mathcal{T}_h$, we denote by $P_s(K)$ the space of polynomials of degree at most $s$ on $K$.  On a face $e = K_1 \cap K_2 \in \mathcal{E}_h$, we denote the jump and average of a piecewise smooth scalar function $f$ by
\[
\llbracket f \rrbracket = f_1 n_1 + f_2 n_2, \quad \{f\} = \frac{f_1+f_2}{2},
\] 
where $f_i = \left. f \right|_{K_i}$, $n_1$ is the normal vector to $e$ pointing from $K_1$ to $K_2$, and similarly for $n_2$.  

We focus on dimension $d=3$ below, and we later comment about dimension $d=2$ in Remark~\ref{remark:dim2}.
In dimension $d=3$, our numerical method will make use of four approximation spaces: a space $U_h^{\dv} \subset H_0(\dv,\Omega)$ for the velocity $u$ and magnetic field $B$, a space $F_h \subset L^2(\Omega)$ for the density $\rho$, a space $Q_h \subset L^2_{\int=0}(\Omega)$ for the pressure $\widetilde{p}$, and an auxiliary space $U_h^{\curl} \subset H_0(\curl,\Omega)$.  For the velocity and magnetic field, we use the Raviart-Thomas space
\[
RT_s(\mathcal{T}_h) = \{ u \in H_0(\dv,\Omega) \mid \left. u \right|_K \in P_s(K)^3 + x P_s(K), \, \forall K \in \mathcal{T}_h \},
\]
where $s \ge 0$ is an integer.  For the pressure, we use the zero-mean subspace of the discontinuous Galerkin space
\[
DG_s(\mathcal{T}_h) = \{ f \in L^2(\Omega) \mid \left. f \right|_K \in P_s(K), \, \forall K \in \mathcal{T}_h \}.
\]
For the density, we use $DG_m(\mathcal{T}_h)$, where $m \ge 0$ is an integer (not necessarily equal to $s$).  For the auxiliary space $U_h^{\curl}$, we use the space of Nedelec elements of the first kind,
\[
NED_s(\mathcal{T}_h) = \{ u \in H_0(\curl,\Omega) \mid \left. u \right|_K \in P_s(K)^3 + x \times P_s(K)^3, \, \forall K \in \mathcal{T}_h \}.
\]
In summary,
\begin{align}
U_h^{\dv} &= RT_s(\mathcal{T}_h), \label{FEvelocity} \\
F_h &= DG_m(\mathcal{T}_h), \label{FEdensity} \\
Q_h &= DG_s(\mathcal{T}_h) \cap L^2_{\int=0}(\Omega), \label{FEpressure} \\
U_h^{\curl} &= NED_s(\mathcal{T}_h). \label{FEelectric}
\end{align}

We define trilinear forms $a_h : L^2(\Omega)^3 \times L^4(\Omega)^3 \times L^4(\Omega)^3 \rightarrow \mathbb{R}$ and $b_h :  L^2(\Omega) \times L^2(\Omega) \times U_h^{\dv} \rightarrow \mathbb{R}$ by
\begin{align}
a_h(w,u,v) &= \int_\Omega w \cdot \nabla \times \pi_h^{\curl} (u \times v) \, dx, \label{ah} \\
b_h(f,g,u) &= -\sum_{K \in \mathcal{T}_h} \int_K (u \cdot \nabla \pi_h f) \pi_h g \, dx + \sum_{e \in \mathcal{E}_h} \int_e u \cdot \llbracket \pi_h f \rrbracket \{\pi_h g\} \, ds,
\end{align}
where $\pi_h^{\curl} : L^2(\Omega)^3 \rightarrow U_h^{\curl}$ and $\pi_h : L^2(\Omega) \rightarrow F_h$ denote the $L^2$-orthogonal projectors onto $U_h^{\curl}$ and $F_h$, respectively.
Note that $b_h$ (restricted to $F_h \times F_h \times U_h^{\dv}$) is a standard discontinous Galerkin discretization of the scalar advection operator~\cite{brezzi2004discontinuous}.  

These trilinear forms possess two important properties that mimic~(\ref{aalternating}-\ref{balternating}).  The trilinear form $a_h$ is alternating in its last two arguments:
\begin{equation} \label{ahalt}
a_h(w,u,v) = -a_h(w,v,u), \quad \forall (w,u,v) \in  L^2(\Omega)^3 \times L^4(\Omega)^3 \times L^4(\Omega)^3.
\end{equation}
Second, using integration by parts, one checks that $b_h$ is alternating in its first two arguments if its last argument is divergence-free:
\begin{equation} \label{bhplusminus}
b_h(f,g,u) = -b_h(g,f,u), \quad \forall (f,g,u) \in L^2(\Omega) \times L^2(\Omega) \times  (U_h^{\dv} \cap \mathring{H}(\dv,\Omega)) .
\end{equation}
Note that $a_h$ does not satisfy a discrete analogue of~(\ref{avanishes}), but there is another choice of $a_h$ which does; see Section~\ref{sec:maghelicity}.

We define our semidiscrete numerical method as follows.  We seek $u,B \in U_h^{\dv}$, $\rho \in F_h$, and $p \in Q_h$ such that
\begin{align}
\langle \partial_t (\rho u), v \rangle + a_h(\rho u, u, v) - a_h(B, B, v) + \frac{1}{2} b_h(u \cdot u, \rho, v) &= \langle p, \dv v \rangle, && \forall v \in U_h^{\dv}, \label{velocityh_cons} \\
\langle \partial_t B, C \rangle + a_h(C,B,u) &= 0, && \forall C \in U_h^{\dv}, \label{magnetich_cons} \\
\langle \partial_t \rho, \sigma \rangle + b_h(\sigma, \rho, u) &= 0, && \forall \sigma \in F_h, \label{densityh_cons} \\
\langle \dv u, q \rangle &= 0, &&\forall q \in Q_h. \label{incompressibleh_cons}
\end{align}

\begin{proposition} \label{prop:div}
The solution of~(\ref{velocityh_cons}-\ref{incompressibleh_cons}) satisfies $\dv u(t) \equiv 0$ for every $t$.  Furthermore, if $B(0)$ is exactly divergence-free, then $\dv B(t) \equiv 0$ for every $t$.
\end{proposition}
\begin{proof}
Since $u \in U_h^{\dv} = RT_s(\mathcal{T}_h)$, we have $\dv u \in DG_s(\mathcal{T}_h)\cap L^2_{\int=0}(\Omega) = Q_h$, so we may take $q = \dv u$ in~(\ref{incompressibleh_cons}).  This shows that $\dv u(t) \equiv 0$ for every $t$.  Since $\nabla \times U_h^{\curl} \subseteq U_h^{\dv}$, equation~(\ref{magnetich_cons}) implies that 
\[
\partial_t B + \nabla \times \pi_h^{\curl} (B \times u) = 0
\]
holds pointwise in $\Omega$.  Taking the divergence of this equation, we see that $\dv B(t) \equiv 0$ for every $t$ if $\dv B(0) \equiv 0$.
\end{proof}

We henceforth assume $\dv B(0) \equiv 0$.

\begin{proposition} \label{prop:cons}
The numerical method~(\ref{velocityh_cons}-\ref{incompressibleh_cons}) exactly preserves $\int_\Omega \rho \, dx$, $\int_\Omega \rho^2 \, dx$, $\int_\Omega \rho u \cdot u + B \cdot B \, dx$, and (if $\rho \equiv 1$) $\int_\Omega u \cdot B \, dx$.
\end{proposition}
\begin{proof}
Since $a_h$ and $b_h$ satisfy~(\ref{ahalt}-\ref{bhplusminus}), and since $u$ and $B$ are divergence-free, the proof is virtually identical to the one given in Section~\ref{sec:intro} for solutions of~(\ref{velocity}-\ref{incompressible}).
\end{proof}

Equations~(\ref{velocityh_cons}-\ref{incompressibleh_cons}) are not implementable in their present form, because they incorporate projections of the test function $v$, e.g., in the term $a_h(\rho u, u, v)$.  The following lemma rectifies this.

\begin{lemma} \label{lemma:rectify}
For every $u, v, B, C \in U_h^{\dv}$ and $\rho \in F_h$, we have
\begin{align}
a_h(\rho u, u, v) &= \langle w \times u, v \rangle, \\
a_h(B,B,v) &= \langle J \times B, v \rangle, \\
a_h(C,B,u) &= \langle \nabla \times E, C \rangle, \\
\frac{1}{2} b_h(u \cdot u, \rho, v) &= b_h(\theta,\rho,v), \label{thetarecast}
\end{align}
where $w,J,E \in U_h^{\curl}$ and $\theta \in F_h$ are defined by
\begin{align}
\langle w, z \rangle &= \langle \rho u, \nabla \times z \rangle, && \forall z \in U_h^{\curl}, \label{wdef} \\
\langle J, K \rangle &= \langle B, \nabla \times K \rangle, && \forall K \in U_h^{\curl}, \\
\langle E, F \rangle &= -\langle u \times B, F \rangle, && \forall F \in U_h^{\curl}, \label{Edef} \\
\langle \theta, \tau \rangle &= \frac{1}{2} \langle u \cdot u, \tau \rangle, && \forall \tau \in F_h. \label{thetadef}
\end{align}
\end{lemma}
\begin{proof}
Since $w \in U_h^{\curl}$, we have
\[
\langle w \times u, v \rangle = \langle w, u \times v \rangle= \langle w, \pi_h^{\curl} (u \times v) \rangle = \langle \rho u, \nabla \times \pi_h^{\curl} (u \times v) \rangle = a_h(\rho u, u, v).
\]
Since $J \in U_h^{\curl}$, we have
\[
\langle J \times B, v \rangle = \langle J, B \times v \rangle = \langle J, \pi_h^{\curl} (B \times v) \rangle = \langle B, \nabla \times  \pi_h^{\curl} (B \times v) \rangle = a_h(B,B,v).
\]
Since~(\ref{Edef}) implies $E = -\pi_h^{\curl}(u \times B)$, we have
\[
\langle \nabla \times E, C \rangle = - \langle \nabla \times \pi_h^{\curl}(u \times B), C \rangle = -a_h(C,u,B) = a_h(C,B,u).
\]
Finally,~(\ref{thetarecast}) follows immediately from~(\ref{thetadef}), since it implies $\theta = \frac{1}{2} \pi_h (u \cdot u)$.
\end{proof} 

We can now restate the numerical method~(\ref{velocityh_cons}-\ref{incompressibleh_cons}) in an equivalent, implementable form.  It seeks $u,B \in U_h^{\dv}$, $\rho,\theta \in F_h$, $p \in Q_h$, and $w,J,E \in U_h^{\curl}$ such that
\begin{align}
\langle \partial_t (\rho u), v \rangle + \langle w \times u, v \rangle - \langle J \times B, v \rangle + b_h(\theta,\rho,v) &= \langle p, \dv v \rangle, && \forall v \in U_h^{\dv}, \label{velocityh_simp} \\
\langle \partial_t B, C \rangle + \langle \nabla \times E, C \rangle &= 0, && \forall C \in U_h^{\dv}, \\
\langle \partial_t \rho, \sigma \rangle + b_h(\sigma,\rho,u) &= 0, && \forall \sigma \in F_h, \\
\langle \dv u, q \rangle &= 0, && \forall q \in Q_h, \\
\langle w, z \rangle &= \langle \rho u, \nabla \times z \rangle, && \forall z \in U_h^{\curl}, \\
\langle J, K \rangle &= \langle B, \nabla \times K \rangle, && \forall K \in U_h^{\curl}, \\
\langle \theta, \tau \rangle &= \frac{1}{2} \langle u \cdot u, \tau \rangle, && \forall \tau \in F_h, \\
\langle E, F \rangle &= -\langle u \times B, F \rangle, && \forall F \in U_h^{\curl}. \label{Eh_simp}
\end{align}

Note that in the above scheme, the trilinear form $b_h$ is evaluated only on $F_h \times F_h \times U_h^{\dv}$, since $\rho,\theta,\sigma \in F_h$.  For these inputs, the projection $\pi_h$ does not appear:
\[
b_h(f,g,u) = -\sum_{K \in \mathcal{T}_h} \int_K (u \cdot \nabla f) g \, dx + \sum_{e \in \mathcal{E}_h} \int_e u \cdot \llbracket f \rrbracket \{ g\} \, ds, \quad \forall (f,g,u) \in F_h \times F_h \times U_h^{\dv}.
\]

\paragraph{The case of constant density.}
For the benefit of the reader, let us record what the scheme~(\ref{velocityh_simp}-\ref{Eh_simp}) reduces to when $\rho \equiv 1$.  In this setting, it seeks $u,B \in U_h^{\dv}$, $p \in Q_h$, and $w,J,E \in U_h^{\curl}$ such that
\begin{align}
\langle \partial_t u, v \rangle + \langle w \times u, v \rangle - \langle J \times B, v \rangle &= \langle p, \dv v \rangle, && \forall v \in U_h^{\dv}, \label{velocityh_rho1} \\
\langle \partial_t B, C \rangle + \langle \nabla \times E, C \rangle &= 0, && \forall C \in U_h^{\dv}, \\
\langle \dv u, q \rangle &= 0, && \forall q \in Q_h, \\
\langle w, z \rangle &= \langle u, \nabla \times z \rangle, && \forall z \in U_h^{\curl}, \\
\langle J, K \rangle &= \langle B, \nabla \times K \rangle, && \forall K \in U_h^{\curl}, \\
\langle E, F \rangle &= -\langle u \times B, F \rangle, && \forall F \in U_h^{\curl}. \label{Eh_rho1}
\end{align}

\begin{proposition}
The numerical method~(\ref{velocityh_rho1}-\ref{Eh_rho1}) exactly preserves $\int_\Omega u \cdot u + B \cdot B \, dx$ and $\int_\Omega u \cdot B \, dx$. Furthermore, $\dv u(t) \equiv 0$ and $\dv B(t) \equiv 0$ for every $t$.
\end{proposition}
\begin{proof}
Apply Propositions~\ref{prop:div}-\ref{prop:cons} with $\rho \equiv 1$.
\end{proof}

\section{A Variant that also Preserves Magnetic Helicity} \label{sec:maghelicity}

If, in place of~(\ref{ah}), we define
\begin{equation}
a_h(w,u,v) = \int_\Omega w \cdot \nabla \times \pi_h^{\curl} (\pi_h^{\curl}u \times \pi_h^{\curl}v) \, dx, \label{ah_comp}
\end{equation}
then we obtain a method that additionally preserves magnetic helicity $\int_\Omega A \cdot B \, dx$, where $A$ is any vector field satisfying $\nabla \times A = B$ and $\left. A \times n\right|_{\partial\Omega} = 0$.  Indeed, we then have the following discrete analogue of~(\ref{avanishes}).
\begin{lemma}
The trilinear form~(\ref{ah_comp}) satisfies
\begin{equation} \label{ahvanishes}
a_h(w,u,v) = 0\; \text{ if } \; \nabla \times w = u.
\end{equation}
\end{lemma}
\begin{proof}
If $\nabla \times w = u$, then we can integrate~(\ref{ah_comp}) by parts and use the fact that $\left.n \times \pi_h^{\curl} (\pi_h^{\curl}u \times \pi_h^{\curl}v)\right|_{\partial\Omega} = 0$ to obtain
\begin{align*}
a_h(w,u,v)
&= \langle w, \nabla \times \pi_h^{\curl} (\pi_h^{\curl}u \times \pi_h^{\curl}v) \rangle & \\
&= \langle \nabla \times w,  \pi_h^{\curl} (\pi_h^{\curl}u \times \pi_h^{\curl}v) \rangle &  \\
&= \langle u,  \pi_h^{\curl} (\pi_h^{\curl}u \times \pi_h^{\curl}v) \rangle & \\
&= \langle \pi_h^{\curl} u,  \pi_h^{\curl}u \times \pi_h^{\curl}v \rangle & \\
&= 0. &
\end{align*}
\end{proof}
The above property of $a_h$ implies that if $A$ is any vector field satisfying $\nabla \times A = B$ and $\left. A \times n\right|_{\partial\Omega} = 0$, then
\begin{align}
\frac{d}{dt} \langle A, B \rangle 
&= \langle \partial_t A, B \rangle + \langle A, \partial_t B \rangle \nonumber \\
&= \langle \partial_t A,  \nabla \times A \rangle + \langle A, \partial_t B \rangle \nonumber \\
&= \langle \nabla \times \partial_t A, A \rangle + \langle A, \partial_t B \rangle \nonumber \\
&= 2\langle \partial_t B, A \rangle \nonumber \\
&= 2\langle \partial_t B, \pi_h^{\dv} A \rangle \nonumber \\
&= -2a_h(\pi_h^{\dv} A, B, u) \nonumber \\
&= -2a_h(A, B, u) \nonumber \\
&= 0. \label{ddtAB}
\end{align}
Above, we used the magnetic field equation~(\ref{magnetich_cons}) with $C=\pi_h^{\dv}A$, and we used with the fact that $a_h(\pi_h^{\dv} A, B, u) = \langle \pi_h^{\dv} A, \nabla \times \pi_h^{\curl} (\pi_h^{\curl}B \times \pi_h^{\curl}u) \rangle = a_h(A, B, u)$ since $\nabla \times U_h^{\curl} \subseteq U_h^{\dv}$. 

Using calculations analogous to those in Lemma~\ref{lemma:rectify}, one finds that when $a_h$ is given by~(\ref{ah_comp}), the method~(\ref{velocityh_cons}-\ref{incompressibleh_cons}) is equivalent to the following method:  Seek $u,B \in U_h^{\dv}$, $\rho,\theta \in F_h$, $p \in Q_h$, and $w,J,H,U,E,\alpha \in U_h^{\curl}$ such that
\begin{align}
\langle \partial_t (\rho u), v \rangle + \langle \alpha, v \rangle + b_h(\theta,\rho,v) &= \langle p, \dv v \rangle, && \forall v \in U_h^{\dv}, \label{velocityh_comp} \\
\langle \partial_t B, C \rangle + \langle \nabla \times E, C \rangle &= 0, && \forall C \in U_h^{\dv}, \label{magnetich_comp} \\
\langle \partial_t \rho, \sigma \rangle + b_h(\sigma,\rho,u) &= 0, && \forall \sigma \in F_h, \label{densityh_comp} \\
\langle \dv u, q \rangle &= 0, && \forall q \in Q_h,  \\
\langle w, z \rangle &= \langle \rho u, \nabla \times z \rangle, && \forall z \in U_h^{\curl}, \\
\langle J, K \rangle &= \langle B, \nabla \times K \rangle, && \forall K \in U_h^{\curl}, \\
\langle \theta, \tau \rangle &= \frac{1}{2} \langle u \cdot u, \tau \rangle, && \forall \tau \in F_h, \label{thetah_comp} \\
\langle H, G \rangle &= \langle B, G \rangle, && \forall G \in U_h^{\curl}, \label{Bproj_comp} \\
\langle U, V \rangle &= \langle u, V \rangle, && \forall V \in U_h^{\curl}, \label{uproj_comp}  \\
\langle E, F \rangle &= -\langle U \times H, F \rangle, && \forall F \in U_h^{\curl}, \label{Eh_comp} \\
\langle \alpha, \beta \rangle &= \langle w \times U - J \times H, \beta \rangle, && \forall \beta \in U_h^{\curl}. \label{alphah_comp}
\end{align}
Note that in comparison with~(\ref{velocityh_simp}-\ref{Eh_simp}), more steps are needed here to remove the projection of the test function $v$ because of the additional projection $\pi ^{\curl} _h$ appearing in~(\ref{ah_comp}).
\begin{proposition} \label{prop:comp}
The numerical method~(\ref{velocityh_comp}-\ref{alphah_comp}) exactly preserves $\int_\Omega \rho \, dx$, $\int_\Omega \rho^2 \, dx$, $\int_\Omega \rho u \cdot u + B \cdot B \, dx$, $\int_\Omega A \cdot B \, dx$, and (if $\rho \equiv 1$) $\int_\Omega u \cdot B \, dx$. 
Furthermore, $\dv u(t) \equiv 0$ and $\dv B(t) \equiv 0$ for every $t$.
\end{proposition}
\begin{proof}
We already showed that $\int_\Omega A \cdot B \, dx$ is preserved.  The proof that $\dv u \equiv 0$ and $\dv B \equiv 0$ is essentially the same as the proof of Proposition~\ref{prop:div}.  The other conservation laws are consequences of the properties~(\ref{ahalt}-\ref{bhplusminus}) of $a_h$ and $b_h$.  
\end{proof}

\paragraph{The case of constant density.}
For the benefit of the reader, let us record what the scheme~(\ref{velocityh_comp}-\ref{alphah_comp}) reduces to when $\rho \equiv 1$.  In this setting, it seeks $u,B \in U_h^{\dv}$, $p \in Q_h$, and $w,J,H,U,E,\alpha \in U_h^{\curl}$ such that
\begin{align}
\langle \partial_t u, v \rangle + \langle \alpha, v \rangle  &= \langle p, \dv v \rangle, && \forall v \in U_h^{\dv}, \label{velocityh_comp_rho1} \\
\langle \partial_t B, C \rangle + \langle \nabla \times E, C \rangle &= 0, && \forall C \in U_h^{\dv}, \\
\langle \dv u, q \rangle &= 0, && \forall q \in Q_h, \\
\langle w, z \rangle &= \langle u, \nabla \times z \rangle, && \forall z \in U_h^{\curl}, \\
\langle J, K \rangle &= \langle B, \nabla \times K \rangle, && \forall K \in U_h^{\curl}, \\
\langle H, G \rangle &= \langle B, G \rangle, && \forall G \in U_h^{\curl}, \\
\langle U, V \rangle &= \langle u, V \rangle, && \forall V \in U_h^{\curl}, \\
\langle E, F \rangle &= -\langle U \times H, F \rangle, && \forall F \in U_h^{\curl}, \label{Eh_comp_rho1} \\
\langle \alpha, \beta \rangle &= \langle w \times U - J \times H, \beta \rangle, && \forall \beta \in U_h^{\curl}. \label{alphah_comp_rho1}
\end{align}

\begin{proposition}
The numerical method~(\ref{velocityh_comp_rho1}-\ref{alphah_comp_rho1}) exactly preserves $\int_\Omega u \cdot u + B \cdot B \, dx$, $\int_\Omega A \cdot B \, dx$, and $\int_\Omega u \cdot B \, dx$.  Furthermore, $\dv u(t) \equiv 0$ and $\dv B(t) \equiv 0$ for every $t$.
\end{proposition}
\begin{proof}
Apply Proposition~\ref{prop:comp} with $\rho \equiv 1$.
\end{proof}

\begin{remark}
For most of the remainder of this paper, we will focus our attention on the scheme~(\ref{velocityh_comp}-\ref{alphah_comp}).  Results and techniques that we develop for~(\ref{velocityh_comp}-\ref{alphah_comp}) carry over easily to~(\ref{velocityh_comp_rho1}-\ref{alphah_comp_rho1}) by setting $\rho \equiv 1$.  The same results and techniques (with the exception of magnetic helicity conservation when $d=3$) carry over easily to (\ref{velocityh_simp}-\ref{Eh_simp}) and~(\ref{velocityh_rho1}-\ref{Eh_rho1}) as well.  In fact, we recover the scheme~(\ref{velocityh_simp}-\ref{Eh_simp}) (respectively,~(\ref{velocityh_rho1}-\ref{Eh_rho1})) from Section~\ref{sec:space} by replacing~(\ref{Eh_comp}-\ref{alphah_comp}) (respectively,~(\ref{Eh_comp_rho1}-\ref{alphah_comp_rho1})) by
\begin{align*}
\langle E, F \rangle &= -\langle u \times B, F \rangle, \quad \forall F \in U_h^{\curl}, \\
\alpha &= w \times u - J \times B.
\end{align*}
\end{remark}

\begin{remark} \label{remark:dim2}
Specializing the above schemes to dimension $d=2$ is straightforward, but one must take care to distinguish between vector fields in the plane ($u,B,H,U$, and $\alpha$) and vector fields orthogonal to it ($w,J$, and $E$).  Accordingly, we identify $w,J$, and $E$ with scalar fields and discretize them with the continuous Galerkin finite element space
\begin{equation} \label{cg}
CG_s(\mathcal{T}_h) = \{ f \in C^0(\overline{\Omega}) \mid \left.f\right|_K \in P_s(K), \, \forall K \in \mathcal{T}_h, \, f=0 \text{ on } \partial\Omega \} \subset H^1_0(\Omega)
\end{equation}
when $d=2$.  We do the same for the test vector fields $z,K$, and $F$.
\end{remark}

\section{Upwinding} \label{sec:upwinding}

To incorporate upwinding into the density advection equation~(\ref{densityh_comp}), one can replace~(\ref{densityh_comp}) by
\begin{align}
\langle \partial_t \rho, \sigma \rangle + b_h(\sigma,\rho,u) + \sum_{e \in \mathcal{E}_h} \int_e \beta_e(u) \llbracket \sigma \rrbracket \cdot \llbracket \rho \rrbracket \, ds  &= 0, && \forall \sigma \in F_h, \label{densityh_comp_upwind0}
\end{align}
where $\{\beta_e\}_{e \in \mathcal{E}_h}$ are nonnegative parameters which may depend on $u$.  A standard choice for $\beta_e$ is~\cite{brezzi2004discontinuous}
\[
\beta_e(u) = c |u \cdot n|,
\]
where $c \in [0,\frac{1}{2}]$, although we have found that the smooth approximation
\[
\beta_e(u) = \frac{2c}{\pi} (u \cdot n) \arctan\left( \frac{u \cdot n}{\varepsilon} \right)
\]
with $\varepsilon>0$ small (e.g. $\varepsilon=0.01$) tends to give better numerical performance in our experiments.
Full upwinding corresponds to the choice $c=\frac{1}{2}$~\cite{brezzi2004discontinuous}.  When $c>0$, this modification of the density advection equation interferes with conservation of $\int_\Omega \rho^2 \, dx$ and $\int_\Omega \rho u \cdot u + B \cdot B \, dx$, but not $\int_\Omega \rho \, dx$ since $\sum_{e \in \mathcal{E}_h} \int_e \beta_e(u) \llbracket 1 \rrbracket \cdot \llbracket \rho \rrbracket \, ds = 0$.  However, there is a simple way to restore energy conservation.  As suggested in~\cite{gawlik2020conservative}, one replaces the momentum equation~(\ref{velocityh_comp}) by
\begin{equation} \label{velocityh_comp_upwind0} 
\langle \partial_t (\rho u), v \rangle + \langle \alpha, v \rangle + b_h(\theta,\rho,v) + \sum_{e \in \mathcal{E}_h} \int_e \beta_e(u) \left( \frac{v \cdot n}{u \cdot n} \right) \llbracket \theta \rrbracket \cdot \llbracket \rho\rrbracket  \, ds = \langle p, \dv v \rangle, \quad \forall v \in U_h^{\dv}. 
\end{equation}
Both~(\ref{densityh_comp_upwind0}) and~(\ref{velocityh_comp_upwind0}) can be written more compactly if we introduce the $u$-dependent trilinear form
\begin{equation} \label{bhupwind}
\widetilde{b}_h(u; f,g,v) = b_h(f,g,v) + \sum_{e \in \mathcal{E}_h} \int_e \beta_e(u) \left( \frac{v \cdot n}{u \cdot n} \right) \llbracket \pi_h f \rrbracket \cdot \llbracket \pi_h g \rrbracket  \, ds.
\end{equation}
In terms of $\widetilde{b}_h$,~(\ref{densityh_comp_upwind0}) and~(\ref{velocityh_comp_upwind0}) read
\begin{align}
\langle \partial_t \rho, \sigma \rangle + \widetilde{b}_h(u; \sigma,\rho,u)  &= 0, && \forall \sigma \in F_h, \label{densityh_comp_upwind} \\
\langle \partial_t (\rho u), v \rangle + \langle \alpha, v \rangle + \widetilde{b}_h(u; \theta,\rho,v) &= \langle p, \dv v \rangle, && \forall v \in U_h^{\dv}. \label{velocityh_comp_upwind}
\end{align}
\begin{proposition}
With the exception of $\int_\Omega \rho^2 \, dx$, all of the invariants listed in Proposition~\ref{prop:comp} are preserved by~(\ref{velocityh_comp}-\ref{alphah_comp}) if one replaces~(\ref{densityh_comp}) and~(\ref{velocityh_comp}) by~(\ref{densityh_comp_upwind}) and~(\ref{velocityh_comp_upwind}), respectively.
\end{proposition}
\begin{proof}
The only nontrivial claim to check is energy conservation.  For this, we simply recall that energy conservation is deduced by taking $v=u$ in~(\ref{velocityh_comp}), $C=B$ in~(\ref{magnetich_comp}), and $\sigma=\theta$ in~(\ref{densityh_comp}).  Since 
\[
\widetilde{b}_h(u; \sigma,\rho,u) = \widetilde{b}_h(u; \theta,\rho,v), \quad \text{ if } \sigma=\theta \text{ and } v=u,
\]
the proof of energy conservation carries over to this setting. 
\end{proof}

\section{Temporal Discretization} \label{sec:temp}

We now describe a temporal discretization of (the upwinded version of)~(\ref{velocityh_comp}-\ref{alphah_comp}) that exactly preserves all of the original invariants of (the upwinded version of)~(\ref{velocityh_comp}-\ref{alphah_comp}).

We use a time step $\Delta t>0$, and we write $u_k$ to denote the value of the discrete solution $u$ at time $t_k = k\Delta t$.  We denote $u_{k+1/2} = (u_k+u_{k+1})/2$, with similar notation for $p$, $B$, and $\rho$.  We also denote
\[
(\rho u)_{k+1/2} = \frac{\rho_k u_k  + \rho_{k+1} u_{k+1}}{2}.
\]

When stepping from time $t_k$ to time $t_{k+1}$, we know the values of $u_k$, $p_k$, $B_k$, and $\rho_k$, and we seek to determine $u_{k+1}$, $p_{k+1}$, $B_{k+1}$, and $\rho_{k+1}$.  The auxiliary variables $w$, $J$, $\theta$, $H$, $U$, $E$, and $\alpha$ play a role in this calculation, but we do not index them with the subscript $k$.  Our time discretization reads
\begin{align}
\left\langle \frac{ \rho_{k+1} u_{k+1}-\rho_k u_k }{ \Delta t }, v \right\rangle + \langle \alpha, v \rangle + \widetilde{b}_h( u_{k+1/2}; \theta, \rho_{k+1/2}, v )  - \langle p_{k+1}, \dv v \rangle &= 0, && \forall v \in U_h^{\dv}, \label{velocitydt} \\
\left\langle \frac{ B_{k+1}-B_k }{ \Delta t }, C \right\rangle + \langle \nabla \times E, C \rangle &= 0 && \forall C \in U_h^{\dv}, \label{magneticdt} \\
\left\langle \frac{\rho_{k+1}-\rho_k}{\Delta t}, \sigma \right\rangle + \widetilde{b}_h( u_{k+1/2}, \sigma, \rho_{k+1/2}, u_{k+1/2} ) &= 0, && \forall \sigma \in F_h, \label{densitydt} \\
\langle \dv u_{k+1}, q \rangle &= 0, && \forall q \in Q_h, \label{incompressibledt}
\end{align}
where $\theta$, $E$, and $\alpha$ (as well as $w$, $J$, $H$, and $U$) are determined from the equations
\begin{align}
\langle w, z \rangle &= \langle (\rho u)_{k+1/2}, \nabla \times z \rangle, && \forall z \in U_h^{\curl}, \label{wdt} \\
\langle J, K \rangle &= \langle B_{k+1/2}, \nabla \times K \rangle, && \forall K \in U_h^{\curl}, \\
\langle \theta, \tau \rangle &= \frac{1}{2} \langle u_k \cdot u_{k+1}, \tau \rangle, && \forall \tau \in F_h, \label{thetadt} \\
\langle H, G \rangle &= \langle B_{k+1/2}, G \rangle, && \forall G \in U_h^{\curl}, \\
\langle U, V \rangle &= \langle u_{k+1/2}, V \rangle, && \forall V \in U_h^{\curl}, \\
\langle E, F \rangle &= -\langle U \times H, F \rangle, && \forall F \in U_h^{\curl}, \\
\langle \alpha, \beta \rangle &= \langle w \times U - J \times H, \beta \rangle, && \forall \beta \in U_h^{\curl}.  \label{alphadt}
\end{align}

Notice that the midpoint rule has been adopted in all equations above except~(\ref{thetadt}), where $u \cdot u$ is discretized as $u_k \cdot u_{k+1}$.  We do this in order to take advantage of the identity
\begin{equation} \label{rhousquared}
\begin{split}
&\frac{1}{2\Delta t} \int_\Omega \left( \rho_{k+1} u_{k+1} \cdot u_{k+1} - \rho_k u_k \cdot u_k \right) \, dx \\
&= \left\langle \frac{\rho_{k+1} u_{k+1} - \rho_k u_k}{ \Delta t }, \frac{u_k+u_{k+1}}{2} \right\rangle - \frac{1}{2} \left\langle \frac{\rho_{k+1}-\rho_k}{\Delta t}, u_k \cdot u_{k+1} \right\rangle
\end{split}
\end{equation}
when proving energy conservation below.

\begin{proposition}
If $\dv B_0 \equiv 0$, then the solution of~(\ref{velocitydt}-\ref{alphadt}) satisfies
\begin{align}
\int_\Omega \rho_{k+1} \, dx &= \int_\Omega \rho_k \, dx \label{massdt} \\
\int_\Omega \rho_{k+1}^2 \, dx &\le \int_\Omega \rho_k^2 \, dx, \text{ with equality if } \beta_e = 0, \, \forall e \in \mathcal{E}_h, \label{rho2intdt} \\
\int_\Omega \rho_{k+1} u_{k+1} \cdot u_{k+1} + B_{k+1} \cdot B_{k+1} \, dx &= \int_\Omega \rho_k u_k \cdot u_k + B_k \cdot B_k \, dx,  \label{energydt} \\
\int_\Omega u_{k+1} \cdot B_{k+1} \, dx &= \int_\Omega u_k \cdot B_k \, dx, \quad \text{ if } \rho_0 \equiv 1, \\
\int_\Omega A_{k+1} \cdot B_{k+1} \, dx &= \int_\Omega A_k \cdot B_k \, dx, \\
\dv u_k &\equiv 0, \\
\dv B_k &\equiv 0
\end{align}
for every $k$.  Here, $A_k$ denotes any vector field satisfying $\nabla \times A_k = B_k$ and $\left. A_k \times n \right|_{\partial\Omega} = 0$.
\end{proposition}
\begin{proof}
Let us rewrite the scheme~(\ref{velocitydt}-\ref{alphadt}) in terms of the trilinear forms~(\ref{ah_comp}) and~(\ref{bhupwind}) using the techniques in Lemma~\ref{lemma:rectify}.  To simpilfy notation, we suppress subscripts on quantities evaluated at step $k+1/2$.  Thus, we abbreviate $u_{k+1/2}$, $B_{k+1/2}$, $\rho_{k+1/2}$, and $(\rho u)_{k+1/2}$ as $u$, $B$, $\rho$, and $\rho u$, respectively.  We also denote $D_{\Delta t} (\rho u) = \frac{\rho_{k+1} u_{k+1} - \rho_k u_k}{\Delta t}$, $D_{\Delta t} B = \frac{B_{k+1}-B_k}{\Delta t}$, etc.   In analogy  with~(\ref{velocityh_cons}-\ref{incompressibleh_cons}), the scheme~(\ref{velocitydt}-\ref{alphadt}) is equivalent to
\begin{align}
\langle D_{\Delta t} (\rho u), v \rangle + a_h(\rho u, u, v) - a_h(B, B, v) & \nonumber\\ + \frac{1}{2} \widetilde{b}_h(u; u_k \cdot u_{k+1}, \rho, v) - \langle p_{k+1}, \dv v \rangle &= 0, && \forall v \in U_h^{\dv}, \label{velocityh_consdt} \\
\langle D_{\Delta t} B, C \rangle + a_h(C,B,u) &= 0, && \forall C \in U_h^{\dv}, \label{magnetich_consdt} \\
\langle D_{\Delta t} \rho, \sigma \rangle + \widetilde{b}_h(u; \sigma, \rho, u) &= 0, && \forall \sigma \in F_h, \label{densityh_consdt} \\
\langle \dv u_{k+1}, q \rangle &= 0, &&\forall q \in Q_h, \label{incompressibleh_consdt}
\end{align}
It is immediate from~(\ref{incompressibleh_consdt}) that $\dv u_k \equiv 0$ for every $k$, since we can take $q=\dv u_{k+1}$ in~(\ref{incompressibleh_consdt}).  In addition, $\dv B_k \equiv 0$ for every $k$ since~(\ref{magneticdt}) and the containment $\nabla \times U_h^{\curl} \subseteq U_h^{\dv}$ imply that
\[
B_{k+1} = B_k - (\Delta t) (\nabla \times E)
\]
holds pointwise, so $\dv B_{k+1} = \dv B_k$.
Next, taking $\sigma = 1$ in the density equation~(\ref{densityh_consdt}) yields
\[
\frac{1}{\Delta t}\int_\Omega (\rho_{k+1}-\rho_k) \, dx = \langle D_{\Delta t} \rho, 1 \rangle = -\widetilde{b}_h(u;1,\rho,u) = 0.
\]
Taking $\sigma = \rho$ in~(\ref{densityh_consdt}) and using~(\ref{bhplusminus}), we deduce that
\begin{align*}
&\frac{1}{2\Delta t} \int_\Omega (\rho_{k+1}^2 - \rho_k^2) \, dx \\&= \left\langle \frac{\rho_{k+1}-\rho_k}{\Delta t}, \frac{\rho_k+\rho_{k+1}}{2} \right\rangle = \langle D_{\Delta t} \rho, \rho \rangle = -\widetilde{b}_h(u;\rho,\rho,u) = -\sum_{e \in \mathcal{E}_h} \int_e \beta_e(u) \llbracket \rho \rrbracket \cdot \llbracket \rho \rrbracket  \, ds \le 0,
\end{align*}
with equality if $\beta_e=0$ for every $e \in \mathcal{E}_h$.
To prove conservation of energy, we use the identity~(\ref{rhousquared}), which reads
\[
\frac{1}{2\Delta t} \int_\Omega \big( \rho_{k+1} u_{k+1} \cdot u_{k+1} - \rho_k u_k \cdot u_k \big) \, dx = \langle D_{\Delta t}(\rho u), u \rangle - \frac{1}{2} \langle D_{\Delta t} \rho, u_k \cdot u_{k+1} \rangle
\]
in our abbreviated notation.
Taking $v=u$ in the momentum equation~(\ref{velocityh_consdt}) and $C=B$ in the magnetic field equation~(\ref{magnetich_consdt}) then gives
\begin{align*}
&\frac{1}{2\Delta t} \int_\Omega \big( \rho_{k+1} u_{k+1} \cdot u_{k+1} + B_{k+1} \cdot B_{k+1} \big) - \big( \rho_k u_k \cdot u_k + B_k \cdot B_k \big) \, dx \\
&= \langle D_{\Delta t}(\rho u), u \rangle - \frac{1}{2} \langle D_{\Delta t} \rho, u \cdot u \rangle + \langle D_{\Delta t} B, B \rangle \\
&= \langle p_{k+1}, \dv u \rangle - a_h(\rho u, u, u) + a_h(B,B,u) - \frac{1}{2}\widetilde{b}_h(u; u_k \cdot u_{k+1}, \rho, u) - \frac{1}{2} \langle D_{\Delta t} \rho, u_k \cdot u_{k+1} \rangle - a_h(B,B,u)  \\
&= - \frac{1}{2}\widetilde{b}_h(u; \pi_h(u_k \cdot u_{k+1}), \rho, u) - \frac{1}{2} \langle D_{\Delta t} \rho, \pi_h(u_k \cdot u_{k+1}) \rangle \\
&= 0.
\end{align*}
Here, we have used the fact that $\dv u = \frac{\dv u_k + \dv u_{k+1}}{2} =  0$, $a_h$ is alternating in its last two arguments, and~(\ref{densityh_consdt}) holds.
If $\rho \equiv 1$, then taking $v=B$ in the momentum equation~(\ref{velocityh_consdt}) and $C=u$ in the magnetic field equation~(\ref{magnetich_consdt}) gives
\begin{align*}
&\int_\Omega (u_{k+1} \cdot B_{k+1} - u_k \cdot B_k) \, dx \\
&= \langle D_{\Delta t} u, B \rangle + \langle u, D_{\Delta t} B \rangle \\
&= \langle p_{k+1}, \dv B \rangle - a_h(u, u, B) + a_h(B,B,B) - \frac{1}{2}\widetilde{b}_h(u; u_k \cdot u_{k+1}, 1, B) - a_h(u,B,u) \\
&= 0.
\end{align*}
The last line above follows from the fact that $\dv B = 0$, $\widetilde{b}_h(u; u_k \cdot u_{k+1}, 1, B) = -\widetilde{b}_h(u; 1, u_k \cdot u_{k+1}, B) = 0$, and $a_h$ is alternating in its last two arguments.  Finally, to prove magnetic helicity conservation, we write
\begin{align*}
\int_\Omega (A_{k+1} \cdot B_{k+1} - A_k \cdot B_k) \, dx
&= \langle D_{\Delta t} A, B \rangle + \langle A, D_{\Delta t} B \rangle.
\end{align*}
The steps leading to~(\ref{ddtAB}) now carry over verbatim to the time-discrete setting, with $D_{\Delta t}$ replacing $\partial_t$.  It follows that $\int_\Omega (A_{k+1} \cdot B_{k+1} - A_k \cdot B_k) \, dx = 0$.
\end{proof}

\paragraph{Implementation.}
To implement~(\ref{velocitydt}-\ref{alphadt}), we used a fixed point iteration similar to the one described in~\cite{hu2020helicity}.  When stepping from time $t_k$ to time $t_{k+1}$, we first initialize $(u_{k+1},p_{k+1},B_{k+1},\rho_{k+1}) = (u_k,p_k,B_k,\rho_k)$ and compute $w$, $J$, $\theta$, $H$, $U$, $E$, and $\alpha$ from~(\ref{wdt}-\ref{alphadt}).  Then we fix all variables except $\rho_{k+1}$ and solve~(\ref{densitydt}) for $\rho_{k+1}$, we fix all variables except $B_{k+1}$ and solve~(\ref{magneticdt}) for $B_{k+1}$, and we fix all variables except $(u_{k+1},p_{k+1})$ and solve~(\ref{velocitydt}) and~(\ref{incompressibledt}) for $(u_{k+1},p_{k+1})$.  (If upwinding is adopted, we also fix $\beta_e(u_{k+1/2})$ in the last step to ensure the system of equations is linear.)  Then we repeat this process until a fixed point is reached.  All of the systems of equations encountered in this process are linear, so each iteration is relatively inexpensive.

\section{Numerical Examples} \label{sec:numerical}

\paragraph{Convergence.}
We tested the convergence of our methods in the following way.
On the two-dimensional domain $\Omega = [-1,1] \times [-1,1]$, we manufactured an analytical solution 
\begin{align}
u(x,y,t) &= \big( \cos t \cos(\pi x/2)  \sin(\pi y/2) + \sin t \sin \pi x \cos \pi y, \label{velocity_man}  \\
&\quad \quad \quad -\cos t \sin(\pi x/2) \cos(\pi y/2) - \sin t \cos \pi x \sin \pi y   \big), \nonumber \\
B(x,y,t) &= \big( -\sin t \cos(\pi x/2)  \sin(\pi y/2) + \cos t \sin \pi x \cos \pi y, \\
&\quad\quad\quad \sin t \sin(\pi x/2) \cos(\pi y/2) - \cos t \cos \pi x \sin \pi y   \big), \nonumber \\
\rho(x,y,t) &= 2 + \cos t \sin \pi x \cos \pi y + \sin t \cos \pi x \sin \pi y, \label{density_man}\\
p(x,y,t) &= \rho(x,y,t) |u(x,y,t)|^2 - 1, \label{pressure_man}
\end{align}
to~(\ref{velocity0}-\ref{IC}) by adding forcing terms to the right-hand sides of~(\ref{velocity0}-\ref{density0}).  
In other words, we numerically solved~(\ref{incompressible0}-\ref{IC}) and
\begin{align*}
\rho(\partial_t u + u \cdot \nabla u) - (\nabla \times B) \times B &= -\nabla p + f_u, \\
\partial_t B - \nabla \times (u \times B) &= f_B, \\
\partial_t \rho + \dv (\rho u) &= f_\rho,
\end{align*}
with $f_u,f_B,f_\rho$ and $u_0,B_0,\rho_0$ chosen to make the solution equal to~(\ref{velocity_man}-\ref{pressure_man}).
Note that for all $t$, the functions~(\ref{velocity_man}-\ref{pressure_man}) satisfy $\dv u = \dv B = 0$ in $\Omega$, $u \cdot n = B \cdot n = 0$ on $\partial \Omega$, and $\int_\Omega p \, dx \, dy = 0$.  We numerically solved~(\ref{velocity0}-\ref{IC}) with the forcing $f_u$, $f_B$, $f_ \rho  $,  on a sequence of uniform triangulations $\mathcal{T}_h$ of $\Omega$ with maximum element diameter $h = 2^{-j}$, $j=1,2,3,4$.  We used finite element spaces $U_h^{\dv} = RT_0(\mathcal{T}_h)$, $F_h = DG_0(\mathcal{T}_h)$, $Q_h = DG_0(\mathcal{T}_h) \cap L^2_{\int=0}(\Omega)$, $U_h^{\curl} = NED_0(\mathcal{T}_h)$, and $CG_0(\mathcal{T}_h)$ for $w,J$, and $E$ (recall Remark~\ref{remark:dim2}).  We used a small time step $\Delta t = 0.0025$ to ensure temporal discretization errors were negligible, and we measured the errors in the numerical solution at time $t=0.5$.  The results for four methods are shown in Table~\ref{tab:h}:~(\ref{velocityh_simp}-\ref{Eh_simp}) with and without upwinding, and~(\ref{velocityh_comp}-\ref{alphah_comp}) with and without upwinding.  In the table, the exact solution is denoted $(u,B,\rho,p)$, and the numerical solution is denoted $(u_h,B_h,\rho_h,p_h)$.  The errors are measured in the $L^2(\Omega)$-norm, which we denote by $\|\cdot\|$ throughout this section.

\begin{table}[t]
\centering
\maketable{mhdconv_o0.dat}
\caption{$L^2$-errors in the velocity, magnetic field, density, and pressure at time $t=0.5$.}
\label{tab:h}
\end{table}

The results in Table~\ref{tab:h} indicate that the $L^2$-errors converge linearly to zero for the method~(\ref{velocityh_simp}-\ref{Eh_simp}), but sublinearly for the method~(\ref{velocityh_comp}-\ref{alphah_comp}).  Upwinding had little to no effect on accuracy in this experiment. 
We suspect that the sublinear convergence of~(\ref{velocityh_comp}-\ref{alphah_comp}) is attributable to the boundary conditions imposed during the projections of $u \in U_h^{\dv}$ and $B \in U_h^{\dv}$ onto $U_h^{\curl}$ in~(\ref{Bproj_comp}-\ref{uproj_comp}).  Indeed, $u$ and $B$ satisfy $u \cdot n = B \cdot n = 0$ on $\partial \Omega$, but their projections $U$ and $H$ onto $U_h^{\curl}$ satisfy $U \times n = H \times n = 0$ on $\partial \Omega$. 

Figure~\ref{fig:2d} shows the initial conditions $\rho(x,y,0)$ and $B(x,y,0)$, as well as the numerical solution $\rho(x,y,t)$, and $B(x,y,t)$ obtained at time $t=0.5$ with $h=2^{-4}$.
\begin{figure}
\centering
\includegraphics[scale=0.1,trim=8in 0in 8in 0in,clip=true]{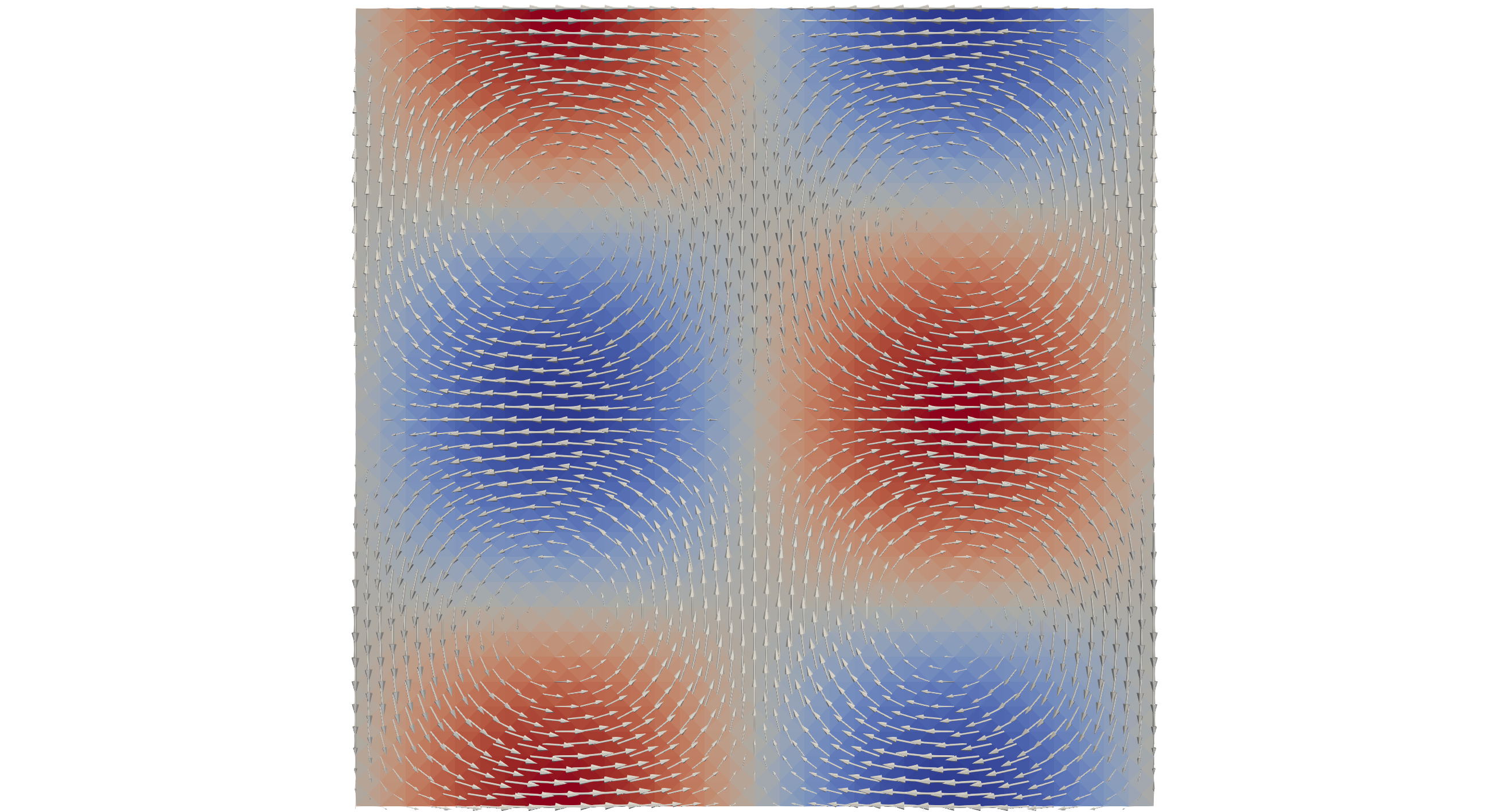}
\includegraphics[scale=0.0845,trim=6in 0in 6in 0in,clip=true]{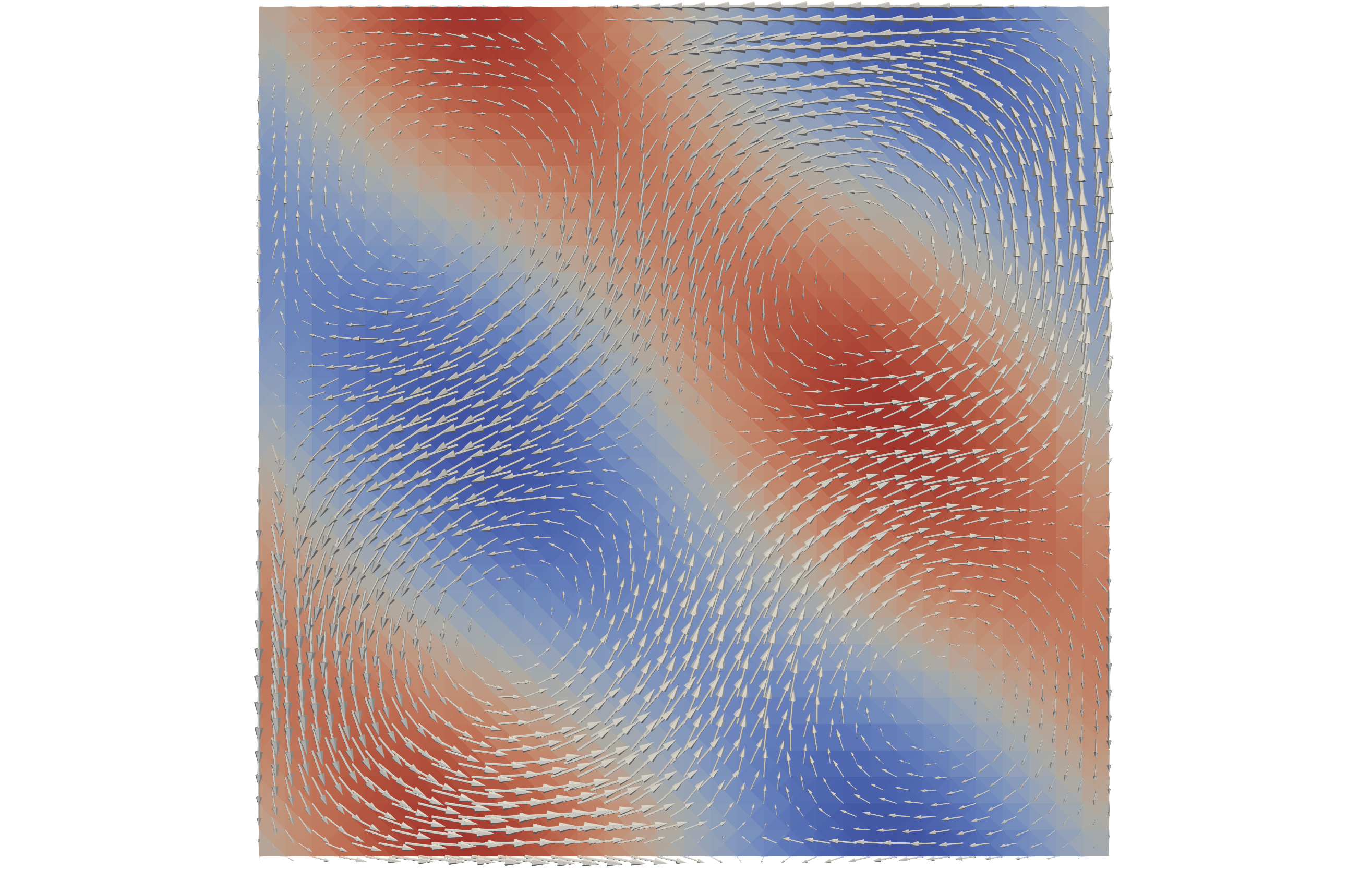}
\includegraphics[scale=0.1,trim=8in 0in 8in 0in,clip=true]{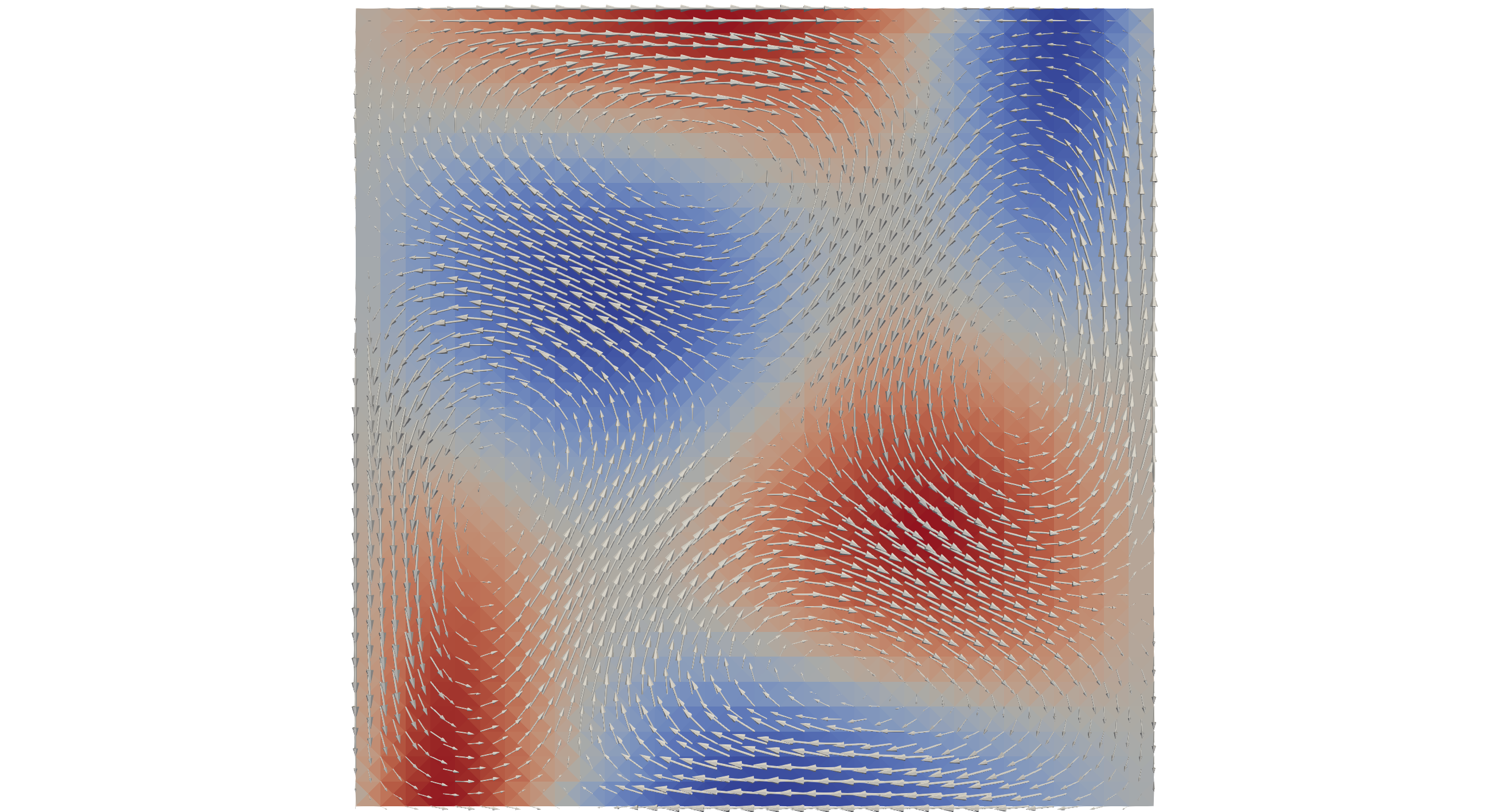}
\caption{Density contours and magnetic field at time $t=0$ (left), at time $t=0.5$ in the presence of forcing (middle), and at time $t=0.5$ in the absence of forcing (right).}
\label{fig:2d}
\end{figure}

\paragraph{Structure preservation.}
To illustrate the structure-preserving properties of our methods, we performed a simulation on the three-dimensional domain $\Omega = [-1,1]^3$  with initial conditions
\begin{align}
u(x,y,z,0) &= \left( y e^{-4(x^2+y^2)}, -x e^{-4(x^2+y^2)}, 0 \right), \label{u03d} \\
B(x,y,z,0) &= \nabla \times \left( (1-x^2)(1-y^2)(1-z^2) v \right), \label{B03d} \\
\rho(x,y,z,0) &= 2 + \sin(xy), \label{rho03d}
\end{align}
where $v = \frac{1}{2}(\sin \pi x, \sin \pi y, \sin \pi z)$.  Note that the vector field $u(x,y,z,0)$ above does not satisfy $u \cdot n = 0$ on $\partial \Omega$; hence, we used the nearest (in the $L^2$-norm) element of $U_h^{\dv} \cap \mathring{H}(\dv,\Omega)$ to $u(x,y,z, 0)$ as our initial condition for $u$ in the simulations.  We used a time step $\Delta t = 0.02$, a uniform triangulation $\mathcal{T}_h$ of $\Omega$ with maximum element diameter $h \approx 0.433$, and finite element spaces $U_h^{\dv} = RT_0(\mathcal{T}_h)$, $F_h = DG_0(\mathcal{T}_h)$, $Q_h = DG_0(\mathcal{T}_h) \cap L^2_{\int=0}(\Omega)$, and $U_h^{\curl} = NED_0(\mathcal{T}_h)$.  Figure~\ref{fig:invariants} plots the evolution of the mass, total squared density, energy, magnetic helicity, divergence of $u$, and divergence of $B$ for four different methods:~(\ref{velocityh_simp}-\ref{Eh_simp}) and~(\ref{velocityh_comp}-\ref{alphah_comp}), each with and without upwinding.  As expected, all of the aforementioned quantities are preserved to machine precision when~(\ref{velocityh_comp}-\ref{alphah_comp}) is used without upwinding.  Upwinding introduces a drift in $\int_\Omega \rho^2 \, dx$, and the use of~(\ref{velocityh_simp}-\ref{Eh_simp}) introduces a drift in the magnetic helicity $\int_\Omega A \cdot B \, dx$.  Here, we computed $A \in U_h^{\curl}$ by solving the (underdetermined) linear system
\[
\langle \nabla \times A, \nabla \times V \rangle = \langle B, \nabla \times V \rangle, \quad \forall V \in U_h^{\curl}.
\]

Note that cross-helicity $\int_\Omega u \cdot B \, dx$ is not plotted in Figure~\ref{fig:invariants} because it is not a conserved quantity of (\ref{velocity0}-\ref{IC}) when $\rho$ is not constant.  To test conservation of cross-helicity, we repeated the above experiment with the initial condition~(\ref{rho03d}) replaced by $\rho(x,y,z,0) = 1$.  The results, plotted in Figure~\ref{fig:invariants1}, show that cross-helicity, energy, magnetic helicity, and the constraints $\dv u = \dv B =0$ are conserved to machine precision by~(\ref{velocityh_comp}-\ref{alphah_comp}), whereas~(\ref{velocityh_simp}-\ref{Eh_simp}) conserves all but magnetic helicity.

\begin{figure}
\centering
\hspace{-0.4in}
\begin{tikzpicture}
\begin{groupplot}[
group style={
    group name=my plots,
    group size=2 by 2,
    xlabels at=edge bottom,
    ylabels at=edge left,
    horizontal sep=2cm,vertical sep=2cm},
    ymode=log,xlabel=$t$,ylabel=Error,
    ylabel style={yshift=-0.1cm},
    ymin = 10^-20,
    ymax = 1,
    legend style={at={(1.0,0.0)},
	anchor=south,
	column sep=1ex,
	text=black},
    legend columns=6
    ]
\nextgroupplot[legend to name=testLegend,title={(\ref{velocityh_simp}-\ref{Eh_simp}), no upwind}]
\makeplot{mhd_n3o0d0r2a0b0_invariants.dat}
\legend{$\int_\Omega \rho \, dx$,$\int_\Omega \rho^2 \, dx$,$\frac{1}{2} \int_\Omega u \cdot u + B \cdot B \, dx$,$\int_\Omega A \cdot B \, dx$,$\|\dv u\|$,$\|\dv B\|$}
\nextgroupplot[title={(\ref{velocityh_comp}-\ref{alphah_comp}), no upwind}]
\makeplot{mhd_n3o0d0r2a1b0_invariants.dat}
\nextgroupplot[title={(\ref{velocityh_simp}-\ref{Eh_simp}), upwind}]
\makeplot{mhd_n3o0d0r2a0b1_invariants.dat}
\nextgroupplot[title={(\ref{velocityh_comp}-\ref{alphah_comp}), upwind}]
\makeplot{mhd_n3o0d0r2a1b1_invariants.dat}
\end{groupplot}
\end{tikzpicture}
\ref{testLegend}
\caption{Errors $|F(t)-F(0)|$ in conserved quantities $F(t)$ during a three-dimensional simulation with variable density.  Results are plotted for four different methods: (\ref{velocityh_simp}-\ref{Eh_simp}) and (\ref{velocityh_comp}-\ref{alphah_comp}), each with and without upwinding.}
\label{fig:invariants}
\end{figure}
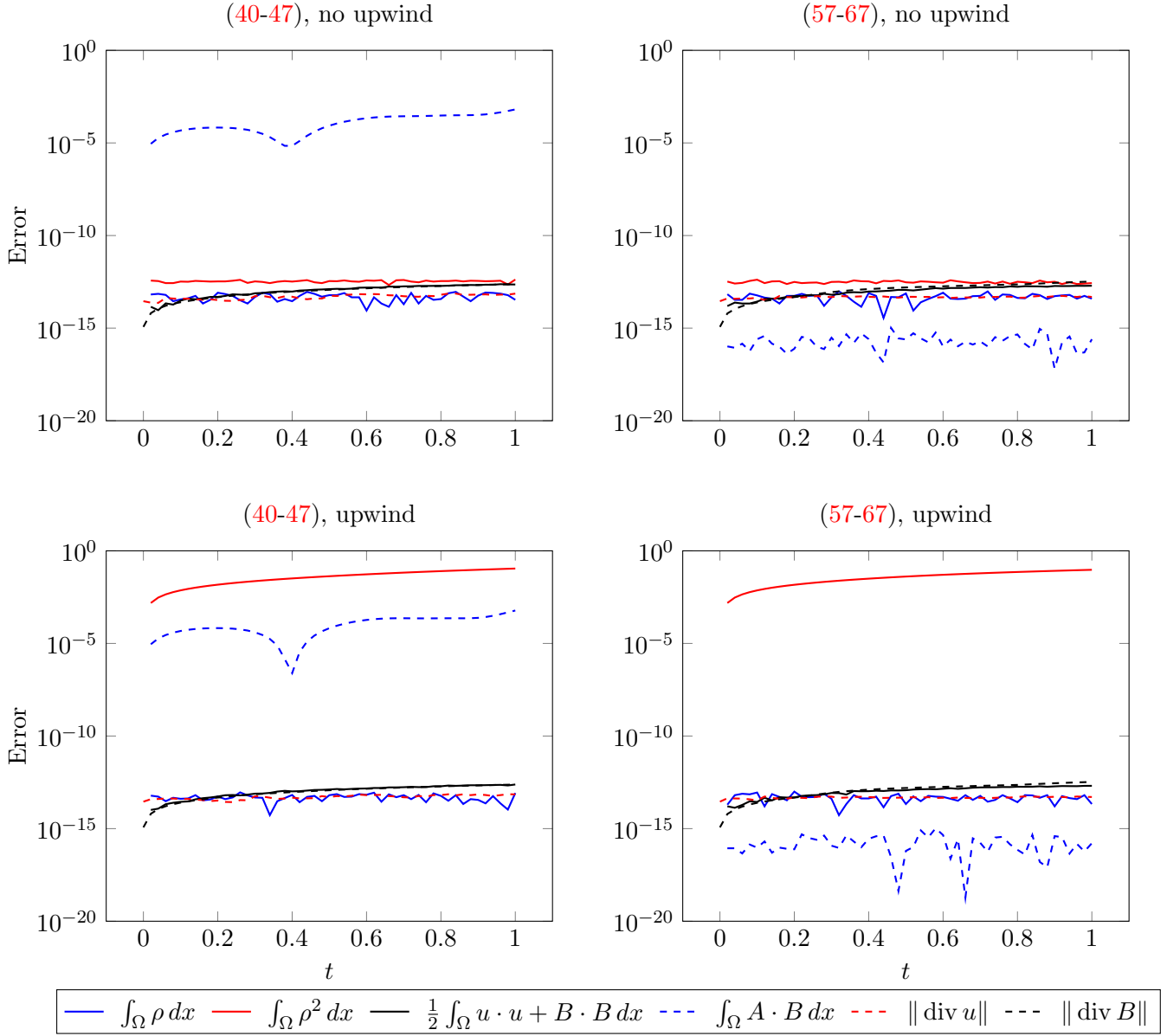

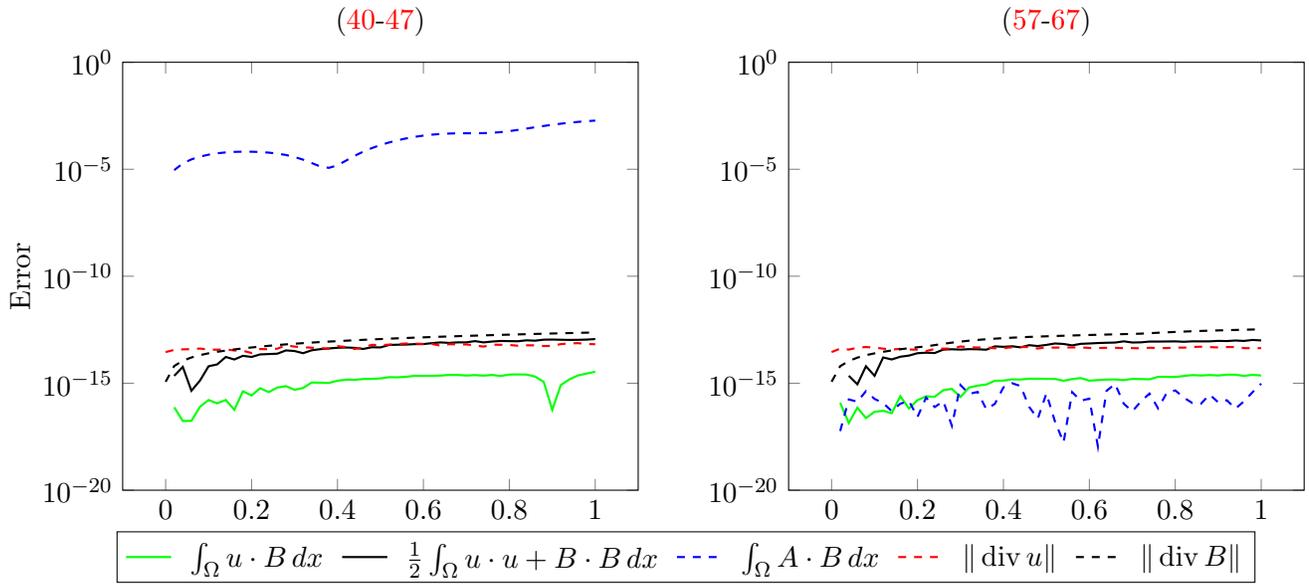
\begin{figure}
\centering
\hspace{-0.4in}
\begin{tikzpicture}
\begin{groupplot}[
group style={
    group name=my plots,
    group size=2 by 2,
    xlabels at=edge bottom,
    ylabels at=edge left,
    horizontal sep=2cm,vertical sep=2cm},
    ymode=log,xlabel=$t$,ylabel=Error,
    ylabel style={yshift=-0.1cm},
    ymin = 10^-20,
    ymax = 1,
    legend style={at={(1.0,0.0)},
	anchor=south,
	column sep=1ex,
	text=black},
    legend columns=5
    ]
\nextgroupplot[legend to name=testLegendb,title={(\ref{velocityh_simp}-\ref{Eh_simp})}]
\makeplotone{mhd_n3o0d0r2a0b0c1_invariants.dat}
\legend{$\int_\Omega u \cdot B \, dx$,$\frac{1}{2} \int_\Omega u \cdot u + B \cdot B \, dx$,$\int_\Omega A \cdot B \, dx$,$\|\dv u\|$,$\|\dv B\|$}
\nextgroupplot[title={(\ref{velocityh_comp}-\ref{alphah_comp})}]
\makeplotone{mhd_n3o0d0r2a1b0c1_invariants.dat}
\end{groupplot}
\end{tikzpicture}
\ref{testLegendb}
\caption{Errors $|F(t)-F(0)|$ in conserved quantities $F(t)$ during a three-dimensional simulation with constant density.  Results are plotted for two different methods: (\ref{velocityh_simp}-\ref{Eh_simp}) and (\ref{velocityh_comp}-\ref{alphah_comp}).} 
\label{fig:invariants1}
\end{figure}

\section{Acknowledgements}

We thank Kaibo Hu for helpful discussions.  EG was partially supported by NSF grants DMS-1703719 and DMS-2012427.  FGB was partially supported by the ANR project GEOMFLUID, ANR-14-CE23-0002-01.

\printbibliography

\end{document}